\documentclass{amsart}
\usepackage{color,subfigure}
\newcommand{\Emax}{E_\text{max}}
\newcommand{\Etot}{E_\text{tot}}
\newcommand{\lambdamin}{\lambda_\text{min}}
\newcommand{\e}{{\rm e}}
\newcommand{\jan}[1]{{#1}}
\newcommand{\leaveout}[1]{}
\DeclareMathOperator{\argmin}{argmin}

\usepackage{all,multirow,multicol}

\begin{document}

\title{Energy minimisation of repelling particles\\ on a toric grid}
\author[Niek Bouman, Jan Draisma, and Johan van
Leeuwaarden]{Niek Bouman$^\star$, Jan Draisma$^\ast$, and Johan van Leeuwaarden$^\bullet$\\ \ \\
Eindhoven University of Technology, \\
Department of Mathematics \& Computer Science, \\
P.O. Box 513, 5600 MB Eindhoven, The Netherlands}
\thanks{$^\star$Supported by Microsoft Research through its PhD
Scholarship Programme}
\thanks{$^\ast$ Supported by a Vidi grant from The
Netherlands Organisation for Scientific Research (NWO); also
Centrum voor Wiskunde en Informatica, Amsterdam}
\thanks{$^\bullet$ Supported by an ERC starting grant and a TOP grant from The
Netherlands Organisation for Scientific Research (NWO)}

\begin{abstract}
We explore the minimum energy configurations of repelling particles
distributed over $n$ possible locations forming a toric grid. We
conjecture that the most energy-efficient way to distribute
$n/2$ particles over this space is to place them in a checkerboard
pattern. Numerical experiments validate this conjecture for reasonable
choices of the repelling force. In the present paper, we prove this
conjecture in a large number of special cases---most notably, when the
sizes of the torus are either two or multiples of four in all dimensions
and the repelling force is a completely monotonic function of the Lee
distance between the particles.
\end{abstract}

\maketitle

\section{Introduction}
The problem of distributing particles uniformly over \jan{a} metric space has proved an inspiring question for mathematicians, and has also attracted the attention of biologists, chemists, and physicists working on crystallography, molecular structures, and electrostatics.

\jan{We} consider an interacting particle system \jan{with}
 a repelling force \jan{between} the particles.
\jan{Under influence} of the force, the particles \jan{arrange
themselves} in the most energy-efficient manner (one that
minimises the total amount of energy in the system). The
rationale is that \jan{repelling} particles  tend to
spread apart until there is as much distance between them as
possible, \jan{which results in some sort of ``uniform
distribution'' of the particles}.

We investigate the case in which the repelling force between
any pair of particles is \jan{given by} a function $f$ of
\jan{their} distance. This
function $f$ can take the form $f(x)=x^{-\alpha}$, with $\alpha>0$. This is a
generalisation of the {\it inverse square law} that is well known in physics, where it
describes the behaviour of forces like electrical charge. It also contains the reduction in
power density of an electromagnetic wave as it propagates through space, known as {\it path
loss}.  Path loss plays a crucial role in the design of wireless networks, for which the
exponent $\alpha$ normally ranges from 2 to 4 (where 2 is for propagation in free space, and 4 is for relatively lossy environments \cite{Jorgen95}). Note that, since $f$ is a decreasing function, we assume that the repelling force becomes stronger when the particles move closer to each other.

Denote the locations of the $N$ particles by
$x_1,\ldots,x_N$. The $f$-energy experienced by particle
$x_i$ is defined as $\sum_{j\neq i}f(\delta(x_i,x_j))$, \jan{where
$\delta$ denotes the metric}. Our
original motivation stems from the area of wireless
networks, in which particles are mobile users, and the
$f$-energy of a user is \jan{its} interference \jan{with}
all other users. When experiencing too much interference, a
user is unable to transmit data \cite{Baccelli09,Gupta00}.
\jan{For} a given maximal sustainable
interference level $T$, a natural question is to find the
maximal packing of users that can be active simultaneously
without violating the threshold $T$, \jan{i.e., while
maintaining $f$-energy at most $T$
for each user. In the case
that we will deal with, this question is
intimately related to the problem of minimising the total
energy among all pairs of particles; see
Section~\ref{sec:Framework}.}

\jan{A lot of research in energy minimisation} focuses on
spheres. In \jan{dimension one}, the uniform distribution on the
circle is simply given by the $N$ equally spaced points
(roots of unity \jan{in the complex plane}), but in higher
dimensions it becomes less obvious on what type of patterns
\jan{the particles arrange themselves}, \jan{and this presents} al sorts of mathematical challenges \cite{Cohn07,Hardin04,Kuijlaars97}.

In this paper, \jan{rather than spheres or other
continuous metric spaces,} we consider a
$d$-dimensional grid as possible locations for the
particles. Further, to avoid boundary effects, we wrap
around the boundaries. \jan{This grid} has a natural embedding on a $d$-dimensional torus, see Figure~\ref{torus}, which is why we call it the $d$-dimensional {\em toric grid}.

Let $n_i$ be the size of the grid in dimension~$i$, so that we have a
total of $n=n_1 \cdots n_d$ possible locations for the
particles. \jan{It seems reasonable to expect that}
the best way, in terms of minimal energy, to distribute
$n/2$ particles over this space is to place them in a {\it checkerboard
pattern}, by which we mean \jan{a} pattern in which exactly one particle is
placed at any two neighbouring points. \jan{Of course, because of the
wrapping around, such a pattern exists only when all $n_i$ are
even.}

\jan{For} any toric grid, \jan{regardless of} its dimensions, there
are only two such configurations (see Figure~\ref{square2}), \jan{and
the} checkerboard pattern is \jan{a} natural \jan{``uniform
distribution''} on a toric grid.  \jan{It is also a maximal-size
{\it independent set} in the graph whose vertices are the possible
locations and whose edges connect locations differing by $\pm 1$ in one
coordinate and not at all in the remaining coordinates.} The Lee distance
between two locations is defined as the length of the shortest path,
in terms of edges, between the corresponding vertices on this graph.
\jan{The following conjecture formalises, to some extent, our intuition
that checkerboard patterns are optimal.}

\begin{conj}[Checkerboard Conjecture]
\jan{For any toric grid with $n_1,\ldots,n_d$ all even, and for any
reasonable function $f$ of the Lee distance, the arrangements of $N=n/2$
particles that minimise the maximal $f$-energy experienced by any of
the particles are the two checkerboard patterns, and only these.}
\end{conj}

\jan{Of course, the validity of this conjecture depends on what
{\em reasonable} means. It will certainly imply
{\em decreasing}, but more conditions will be necessary.}

For the example of wireless \jan{networks} the Checkerboard Conjecture
has the following implication. When the threshold $T$ \jan{is}
increased slowly, an increasing number of users is allowed to be active
simultaneously, and there will be a critical value $T=T^*$ at which
this number becomes precisely $n/2$. At that critical threshold,
we then know \jan{that $n/2$ active users can only be arranged in
a checkerboard pattern}---a phenomenon that we have first observed
in simulations of wireless networks. For a large toric grid it is
of course infeasible
to prove the Checkerboard Conjecture by exhaustive search through all
${\binom{n}{N}}$ possible configurations.

\begin{figure}
  \centering
  \includegraphics[width=1.5in,keepaspectratio]{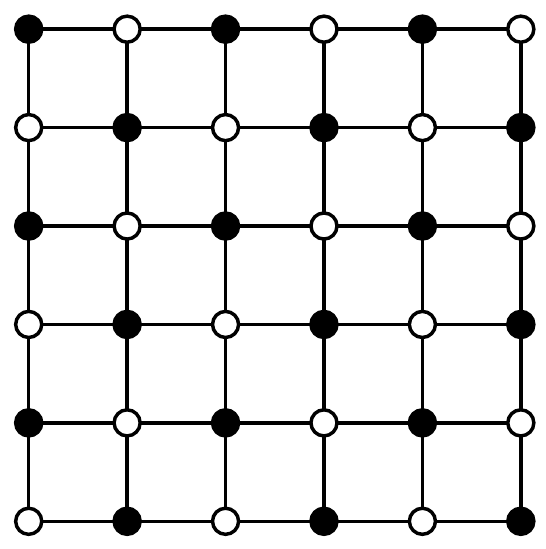}
  \caption{A $6\times6$-grid, the black and white points denote the checkerboard configurations.}
  \label{square2}
\end{figure}

\begin{figure}
  \centering
  \includegraphics[width=3in,keepaspectratio]{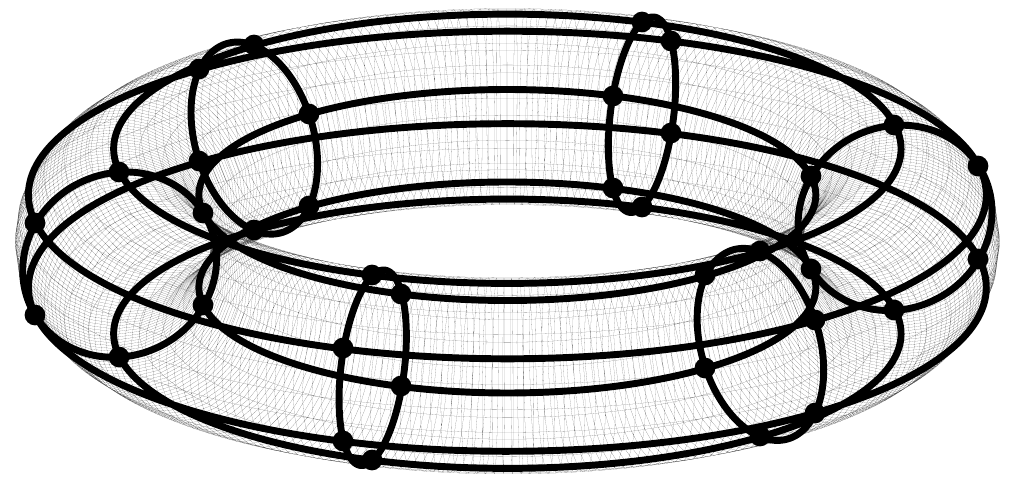}
  \caption{An embedding of a $6\times6$-grid on a $2$-dimensional torus.}
  \label{torus}
\end{figure}

\jan{In this paper we prove the Checkerboard Conjecture in a large number of
special cases. We
first derive a continuous relaxation of our energy minimisation
problem in Section~\ref{sec:Framework}.} Then, as a first \jan{(and
known)} special case, we \jan{prove} the Checkerboard Conjecture for the
one-dimensional toric grid and $f(x)=1/x$ in Section~\ref{sec:Circle}. In
Section~\ref{sec:Hypercube} we prove the Checkerboard Conjecture
for the toric grid that has size two in any dimension and functions
\jan{$f$} whose $k$-th forward difference has sign $(-1)^k$.
\jan{Section~\ref{sec:Four} then gives our strongest result (which however does not
immediately imply any of the previous two results): a proof
for the Checkerboard Conjecture in the case where each $n_i$ is either two or
a multiple of four}, now assuming that \jan{$f$} is \jan{completely
monotonic}; see Section~\ref{sec:Four} for this notion. Finally, \jan{in
Section~\ref{sec:Open} we discuss} some problems open for future research.

\section{Framework} \label{sec:Framework}
In this section we formulate a general version of our energy minimisation
problem, which belongs to the realm of discrete optimisation. We then
derive a continuous relaxation of that problem. If \jan{all} optimal solutions
of the relaxation happen to be \jan{feasible} for the discrete
problem, then they are also the optimal solutions for the discrete
problem. This \jan{is} our strategy for proving special cases
of the Checkerboard Conjecture in the following sections.

Let $d$ be a positive integer, \jan{the {\em dimension}}
of our toric grid. Let $n_1,\ldots,n_d$ be positive integers, and set
$G:=\prod_{i=1}^d (\ZZ/n_i)$, a finite Abelian group. Let
$\Delta:=\{(g,g) \mid g \in G\}$ be the
diagonal in $G \times G$.  Let $u$ be a function $(G \times G) \setminus
\Delta \to \RR$ that is {\em symmetric} ($u(g,h)=u(h,g)$ for all $g,h \in
G$) and {\em $G$-invariant} ($u(k+g,k+h)=u(g,h)$ for all $g,h,k \in G$).
Later we will add the restriction on $u$ that it is a suitable (and in
particular decreasing) function of some $G$-invariant distance on $G$,
so that we may think of $u(g,h)$ as a repelling force between particles
located at $g$ and $h$. But for deriving our relaxation this restriction
is not needed.

Next let $p$ be a natural number less than or equal to
$|G|$. \jan{Let $S$ be a subset} of $G$ of cardinality $p$;
\jan{this is the set of locations of our particles.
Given} $h \in S$ we define
\[ E_h(S):=\sum_{g \in S \setminus \{h\}} u(h,g), \]
called the {\em energy of $S$ \jan{experienced by} $h$}, and
\[ \Emax(S):=\max_{h \in S} E_h(S), \]
called the {\em maximal energy} of $S$.  \jan{The Checkerboard Conjecture
concerns the following optimisation problem with $p=|G|/2$.}

\begin{prb}[Maximal energy minimisation]
Minimise the maximal energy $\Emax(S)$ over all subsets $S \subseteq G$
of cardinality $p$.
\end{prb}

Further we define
\[ \Etot(S):=\sum_{h \in S} E_h(S), \]
called the {\em total energy} of $S$. We can now formulate a second optimisation problem.

\begin{prb}[Total energy minimisation]
Minimise the total energy $\Etot(S)$ over all subsets $S \subseteq G$
of cardinality $p$.
\end{prb}

An important relation between these two optimisation \jan{problems} is
the following. If a set $S^*$ minimises the total energy and happens
to have the property that $E_h(S^*)=E_g(S^*)$ for all $h,g \in S^*$,
then $S^*$ also minimises the maximal energy---indeed, any $S$ with
$\Emax(S) < \Emax(S^*)$ would necessarily have $\Etot(S) < \Etot(S^*)$.
\jan{The condition that $E_h(S^*)$ does not depend on $h \in S^*$ is
satisfied}, in particular, when $S^*$ is a coset of a subgroup \jan{$G'$
of $G$. Indeed, the set $S^*\setminus \{h\}$ is then equal to $G'$ for all $h \in S^*$,
and hence $E_h(S)=\sum_{g \in G'} u(h,h+g)=\sum_{g \in G'} u(0,g)$,
independently of $h \in S^*$, by the $G$-invariance of $u$.} Checkerboard
\jan{patterns} $S$ are examples of cosets. In view of this
relation between maximal energy minimisation and total
energy minimisation, most of
our paper \jan{focuses} on total energy minimisation.

We \jan{proceed to derive a continuous relaxation},
\jan{which} uses the space $V=\RR^G$ of real-valued functions on $G$. We
\jan{represent} a subset $S \subseteq G$ by its {\em characteristic vector}
$x_S \in V$ defined by
\[ x_S(g)=\begin{cases}
1 & \text{if $g \in S$, and}\\
0 & \text{if $g \not \in S$}.
\end{cases}
\]
The space $V$ has two important additional structures. First, it is
equipped with a $G$-action defined by $(g x)(h)=x(h-g)$ for $g,h \in G$
and $x \in V$; \jan{here the minus sign is customary but not strictly
necessary}.
Second, it has a natural \jan{inner product $(x|y):=\sum_{g \in G}
x(g)y(g)$ that is $G$-invariant in the sense that $(hx|hy)=(x|y)$ for
all $h \in G$}.

\jan{Now we define the key quantity in this paper: the {\em energy kernel}
\[ A:=\sum_{g \in G, g \neq 0} u(g,0) g, \]
which lives in the group algebra $\RR G$ and acts
by a linear map on $V$:
\[ (Ax)(h)=\sum_{g \in G, g \neq 0} u(g,0) x(h-g). \]
Alternatively, one may think of $A$ as the symmetric $G \times G$-matrix
whose $(g,h)$-entry equals $u(g,h)$, and perhaps this conceptually
simplifies the following computations.

For any subset $S$ of $G$ we have
\begin{align*}
(x_S|Ax_S) &= \sum_{h \in G} x_S(h) (Ax_S)(h)
= \sum_{h \in S} (Ax_S)(h)\
= \sum_{h \in S} \sum_{g \in G, g \neq 0} u(g,0) x_S(h-g)\\
&= \sum_{h \in S} \sum_{g \in G, g \neq h} u(h-g,0) x_S(g)
= \sum_{h \in S} \sum_{g \in S, g \neq h} u(h-g,0) \\
&= \sum_{h \in S} \sum_{g \in S, g \neq h} u(h,g)
= \Etot(x_S).
\end{align*}

Thus we have identified the total energy of a set $S$ as the inner
product of $x_S$ with its image under the energy kernel. Moreover,
if $S$ has cardinality $p$, then $x_S$ has the further properties
$(x_S|x_S)=(x_S|\one)=p$, where $\one$ is the all-one vector in $V$.}
This motivates the following optimisation problem.

\begin{prb}[Fractional energy minimisation]
Minimise the ``fractional total energy'' $(x|Ax)$ over all $x \in V$
satisfying the constraints $(x|x)=(x|\one)=p$.
\end{prb}

Note that if we add the additional constraint that $x$ be a $0/1$-vector,
then this problem is just total energy minimisation. Thus a feasible
solution of the relaxation can be thought of as a ``fractional'' solution
to the original problem, though this is slight abuse of language for two
reasons. First, the optimal solution to the problem just stated may very
well have irrational coordinates, so that it is not a
``fraction''. Second,
it may have coordinates that are not in the interval $[0,1]$, so that
it cannot be ``rounded'' to an actual solution of the original problem.
We could have added the additional constraint that all coordinates lie in
that interval. This would have led to a more general {\em quadratically
constrained quadratic program}. However, as we \jan{describe} next, the
formulation just given allows for a simple solution once the eigenvalues
of $A$ are known. But first we discuss an example to make all notions
more concrete.

\begin{ex} \label{ex:4by4a}
Let $d=2$ and $n_1=n_2=4$, so that $G=(\ZZ/4)\times(\ZZ/4)$. We depict
a vector $x \in V$ as a $4 \times 4$-matrix of real numbers, with rows
and columns labelled by the elements $0,1,2,3$ of $\ZZ/4$ and entry
$x(k,l)$ at position $(k,l)$. Here is an example of the action of $G$
on $V$ \jan{(``shift $1$ down and $2$ to the right'')}:
\[
(1,2)
\begin{bmatrix}
1 & 2 & 3 & 4 \\
5 & 6 & 7 & 8 \\
9 & 10 & 11 & 12 \\
13 & 14 & 15 & 16
\end{bmatrix}
=
\begin{bmatrix}
15 & 16 & 13 & 14 \\
3 & 4 & 1 & 2 \\
7 & 8 & 5 & 6 \\
11 & 12 & 9 & 10
\end{bmatrix}.
\]
In this case the Lee distance $\delta$ on $G$ is given by
\[ \delta((j,k),(l,m))= |l-j|+|m-k| \]
where $|a|$ is the smallest non-negative integer in $(a+4\ZZ) \cup
(-a+4\ZZ)$. Define $u(g,h)=\frac{1}{\delta(g,h)}$, so that
particles situated at $g$ and $h$ repel each other with a force inverse
proportional to the Lee distance between $g$ and $h$. Take $p=4$ and
$S=\{(0,0),(0,1),(0,2),(0,3)\}$.  Then the energy kernel $A$ satisfies, for
instance,
\begin{align*}
A x_S&=
A
\begin{bmatrix}
1 & 1 & 1 & 1\\
0 & 0 & 0 & 0\\
0 & 0 & 0 & 0\\
0 & 0 & 0 & 0
\end{bmatrix}\\
&=
\begin{bmatrix}
0 + 1 + \frac{1}{2}  +1  &1+   0  + 1  +\frac{1}{2}&\frac{1}{2}+  1  + 0  + 1  &1+\frac{1}{2} + 1   + 0  \\
1 + \frac{1}{2} + \frac{1}{3}+\frac{1}{2}&\frac{1}{2}+ 1  + \frac{1}{2}+\frac{1}{3}&\frac{1}{3}+ \frac{1}{2} + 1  + \frac{1}{2}&\frac{1}{2}+  \frac{1}{3} + \frac{1}{2} + 1  \\
\frac{1}{2}+\frac{1}{3} + \frac{1}{4}+\frac{1}{3}&\frac{1}{3}+\frac{1}{2} + \frac{1}{3}+\frac{1}{4}&\frac{1}{4}+ \frac{1}{3} + \frac{1}{2}+ \frac{1}{3}&\frac{1}{3}+  \frac{1}{4} + \frac{1}{3} + \frac{1}{2}\\
1 + \frac{1}{2} + \frac{1}{3}+\frac{1}{2}&\frac{1}{2}+ 1  + \frac{1}{2}+\frac{1}{3}&\frac{1}{3}+\frac{1}{2} + 1  + \frac{1}{2} &\frac{1}{2}+  \frac{1}{3} + \frac{1}{2} + 1
\end{bmatrix},
\end{align*}
and we find that $(x|Ax)$ equals $4(0+1+1/2+1)=10$, which is indeed the
total energy of $S$. Note that one could represent $A$ by a symmetric
$16 \times 16$-matrix.
\end{ex}

\jan{The geometric intuition behind what follows is that, to minimise
$(x|Ax)$ over the sphere defined by the constraints $(x|x)=(x|\one)=p$,
we have to maximise the component of $x$ in the eigenspace of $A$
corresponding to the smallest eigenvalue of $A$.  Indeed, since $A$
``is'' a symmetric matrix, or, equivalently, a self-adjoint linear map
with respect to $(.|.)$, we know beforehand that all of its eigenvalues
on $V$ are real. Nevertheless, to determine those eigenvalues it is
convenient first to complexify $V$ to $V_\CC:=\CC \otimes V=\CC^G$,
because this allows us to simultaneously diagonalise all group elements
$g$ in their action on $V$. As a consequence, their linear combination
$A$ is then also diagonalised.

We use basic terminology from the representation theory of finite
Abelian groups, for which we refer to \cite[Part 1]{Serre77}.  Since $G$
is Abelian, $V_\CC$ splits as a direct sum of simultaneous eigenspaces
of all $g \in G$, that is, we have}
\[ V_\CC=\bigoplus_{\chi \in G^\vee} V_{\chi}, \]
where $\chi$ runs over
\jan{the set $G^\vee$ of group homomorphisms $(G,+) \to
(\CC,\cdot)$ (called the {\em character group} of $G$)}, and where
$V_\chi$ is the subspace of $V$ defined as
\[ V_\chi := \{v\in V_\CC \mid g v=\chi(g) v \text{ for all } g \in
G\}. \]
\jan{Distinct $V_\chi$ are orthogonal with respect to the complexification
of $(.|.)$ to a Hermitian inner product, and as $V_\CC$ is
isomorphic to the {\em regular representation} of $G$, each $V_\chi$
is one-dimensional.}
\jan{More explicitly, being a group homomorphism, $\chi \in G^\vee$ is
determined by its values $\zeta_1,\ldots,\zeta_d$ on the generators
$(1,0,\ldots,0), \ldots, (0,\ldots,0,1)$ of $G$. Each $\zeta_i$ is an $n_i$-th root
of unity, and the space $V_\chi$ is spanned by the element of $V_\CC$
whose value at $(g_1,\ldots,g_d)$ equals $\zeta_1^{-g_1}
\cdots \zeta_d^{-g_d}$. Note that this function is just the
multiplicative inverse $\chi^{-1}$ of $\chi$, which equals the complex conjugate
$\overline{\chi}$ since the $\zeta_i$ lie on the unit
circle.}

Since $A$ \jan{is a linear combination of the group elements, it acts
by scalar multiplication on each $V_\chi$, and the scalar equals}
\[ \lambda(\chi):=\sum_{g \neq 0} u(g,0) \chi(g). \]
Note that this is the (discrete) {\em Fourier transform} of the
function $G \to \CC,\ g \mapsto u(g,0)$.
\jan{As noted before, this is} a real number, since
\[ \overline{\lambda(\chi)}=
\sum_{g \neq 0} u(g,0) \overline{\chi(g)}=
\sum_{g \neq 0} u(g,0) \chi(-g)=
\lambda(\chi), \]
\jan{where the second equality follows from the fact that
$\overline{\chi(g)}=\chi(g)^{-1}=\chi(-g)$, and}
where we have used the symmetry and $G$-invariance of $u$ in
the last step.
This computation also shows that \jan{$\lambda(\overline{\chi})$
equals $\lambda(\chi)$}. Thus the real
vector space $V$ splits as a direct sum
\[ V=\bigoplus_{ \{\chi,\overline{\chi}\}}
V_{\{\chi,\overline{\chi}\}}\]
over unordered pairs of a character and its conjugate (which
may coincide), where $V_{\{\chi,\overline{\chi}\}}$ is the
real space of real-valued vectors in the complex space $V_\chi +
V_{\overline{\chi}}$. Here the direct sum is in fact orthogonal, each
summand is one-dimensional or two-dimensional according
\jan{to $\chi$ being
real or non-real}, and each summand has corresponding $A$-eigenvalue
$\lambda(\chi)$. As a special case, the trivial character
$\one$
sending every element of $G$ to $1$ is real, and the vector space
$V_{\{\one\}}$ is spanned by the all-one vector $\one$.
Now we can solve fractional energy minimisation as follows.

\begin{prop} \label{prop:Eigenvalue}
Let $\lambdamin$ be the smallest value of $\lambda(\chi)$ as $\chi$
ranges over the non-trivial characters. Then the set of optimal solutions
of fractional energy minimisation consists of all vectors of
the form $\frac{p}{|G|} \one + y$ where $y$
belongs to the eigenspace of $A$ in $V$ with eigenvalue $\lambdamin$,
\jan{is perpendicular to $\one$},
and satisfies $(y|y)=p-{p^2}/{|G|}$.
\end{prop}

Note that the condition that $y$ be perpendicular to $\one$ is automatic
if $\lambdamin<\lambda(\one)$, because \jan{eigenvectors
corresponding to distinct eigenvalues are
perpendicular}. In this case the set of optimal solutions
forms a sphere \jan{whose} dimension \jan{equals} the multiplicity of $\lambdamin$ minus one,
embedded in the affine hyperplane where $(x|\jan{\one})=p$. If $\lambdamin$
happens to equal $\lambda(\one)$, then the dimension of that sphere
of optimal solutions \jan{equals} the multiplicity of $\lambdamin$ minus two.

\begin{proof}
Consider any feasible solution $x$. The condition $(x|\one)=p$ means that
$x$ can be written as $\frac{p}{|G|} \one$ plus $y$ with $y$ in the direct
sum of all $V_{\{\chi,\overline{\chi}\}}$ with $\chi$ non-trivial. The
condition that $(x|x)=p$ translates in the above condition on the norm
of $y$. Let $y_{\{\chi,\overline{\chi}\}}$ be the component
of $y$ in $V_{\{\chi,\overline{\chi}\}}$. Then, by
orthogonality,
\[ (x|Ax)=(\frac{p}{|G|} \one + y | A(\frac{p}{|G|} \one + y))
=
\lambda(\one) \frac{p^2}{|G|} +
\sum_{\{\chi,\overline{\chi}\} \text{ non-trivial}}
\lambda(\chi) ||y_{\{\chi,\overline{\chi}\}}||^2.
\]
For this expression to be minimal among all $y$ with squared
norm $p-{p^2}/{|G|}$ it is necessary and sufficient that
all components $y_{\{\chi,\overline{\chi}\}}$ for which
$\lambda(\chi)$ is not equal to $\lambdamin$ are zero.
This proves the proposition.
\end{proof}

\begin{ex}[Continuation of Example~\ref{ex:4by4a}]
\label{ex:4by4b}
In this case \jan{a} character $\chi \in G^\vee$ is determined by its
values $\zeta_1:=\chi(1,0)$ and $\zeta_2:=\chi(0,1)$, and these complex
numbers must satisfy $\zeta_1^4=\zeta_2^4=1$. Conversely, any such pair
determines a character, \jan{which maps the pair
$(g_1,g_2) \in (\ZZ/4)^2$ to $\zeta_1^{g_1} \zeta_2^{g_2}$}.
The eigenvalue $\lambda(1,i)$ of $A$ equals
(reading elements row-wise)
\[ (0 + i - \frac{1}{2} - i) + (1 + \frac{i}{2} - \frac{1}{3} - \frac{i}{2}) + (\frac{1}{2} + \frac{i}{3} - \frac{1}{4} - \frac{i}{3}) +
(1 + \frac{i}{2} - \frac{1}{3} - \frac{i}{2}) =
\frac{13}{12}.
\]
The eigenvalues of the energy
kernel $A$ are as follows.
\begin{center}
\begin{tabular}{cc|cccc}
\multicolumn{2}{c}{} & \multicolumn{4}{c}{$\zeta_2$}\\
\multicolumn{2}{c|}{$12 \cdot \lambda(\zeta_1,\zeta_2)$}& $1$ & $i$ &
$-1$ & $-i$\\
\hline
\multirow{4}{*}{$\zeta_1$}
& $1$  & 103 & 13 & -9 & 13 \\
& $i$  & 13  & -9 & -19& -9 \\
& $-1$ & -9  &-19 & -25& -19\\
& $-i$ & 13  & -9 & -19&-9
\end{tabular}\\
\end{center}
This shows that $\lambda(1,1)$ is the maximal eigenvalue, with
a one-dimensional eigenspace spanned by $\one$; and
$\lambda(-1,-1)$ is the minimal eigenvalue, with a one-dimensional
eigenspace spanned by
\[
z=\jan{((g_1,g_2) \mapsto (-1)^{g_1} \cdot (-1)^{g_2})}=
\begin{bmatrix}
1 & -1 & 1 & -1\\
-1&  1 & -1&  1\\
1 & -1 & 1 & -1\\
-1&  1 & -1&  1
\end{bmatrix}.
\]
Hence fractional energy minimisation has two optimal solutions (forming
a zero-dimensional sphere). For $p=4$ these are $\frac{1}{4} \one \pm
\frac{\sqrt{3}}{4}z$. Both of these have coordinates in $[0,1]$, but
neither of these is integral or even rational.  Of course, a brute-force
\jan{computer}
search over all configurations of $p=4$ particles on $G$ is possible in
this case. Such a search reveals that, up to translations, the following
three characteristic vectors are the ones minimising the total energy:
\[
\begin{bmatrix}
1 & 0 & 1 & 0\\
0 & 0 & 0 & 0\\
0 & 1 & 0 & 1\\
0 & 0 & 0 & 0
\end{bmatrix},\
\begin{bmatrix}
1 & 0 & 0 & 0\\
0 & 0 & 1 & 0\\
1 & 0 & 0 & 0\\
0 & 0 & 1 & 0
\end{bmatrix},\
\begin{bmatrix}
1 & 0 & 0 & 0\\
0 & 1 & 0 & 0\\
0 & 0 & 0 & 1\\
0 & 0 & 1 & 0
\end{bmatrix}.
\]
Note that the first two are each other's images under swapping rows and
columns, and that they represent the cyclic subgroups of $G$ generated
by $(2,1)$ and $(1,2)$, respectively. The last one is not a coset of a
subgroup, but still has the property that the energies \jan{experienced
by} all particles are the same. This means that these configurations
also minimise the maximal energy.  We conclude that for $p=4$ the optimal
solution to the fractional energy minimisation problem has little bearing
on the original discrete optimisation problem, \jan{and it would be very
interesting to find a strengthening of our relaxation that does give
the optimal solutions for $p=4$.}

On the other hand, \jan{consider now $p=8=|G|/2$.} Then the optimal
solutions for fractional energy minimisation are $\frac{1}{2} \one \pm
\frac{1}{2} z$, and these are exactly the characteristic vectors of the
two checkerboard patterns. \jan{Since fractional energy minimisation is a
relaxation of total energy minimisation, this proves that the checkerboard
patterns also minimise total energy $\Etot$. And since they are cosets
of a subgroup, the argument given after the introduction of $\Etot$
shows that the checkerboard patterns are also the (unique) minimisers of
$\Emax$. We have thus {\em proved} our first instance of the Checkerboard
Conjecture. In the following three sections we follow exactly the
same strategy: we determine the minimal eigenvalue of the energy kernel
to determine the optimal solutions of fractional energy minimisation,
and when these are the checkerboard patterns, then we are done. This
works in quite a number of cases---in particular, when each $n_i$ is
either two or a multiple of four, see
Section~\ref{sec:Four}---but not in general,
see Section~\ref{sec:Open}}.
\end{ex}

We conclude this section by relating our techniques and results to
existing literature on energy minimisation. The most relevant reference
is \cite{Cohn07}, which is primarily concerned with energy minimisation of repelling
particles on high-dimensional spheres. There are close
parallels between our approach and the techniques there. In particular,
both papers make essential use of the representation theory of the natural
symmetry group of the problem. In our case this is the group $G$ itself,
while in the spherical case it is the orthogonal group. The fact that
$G$ is Abelian (and hence diagonalisable) makes the representation
theory much easier than that required for the spherical problem. On
the other hand, there are also important distinctions. First of all,
\cite{Cohn07} makes essential use of the fact that the symmetry group acts
distance-transitively, while our group $G$ acts only simply transitively
on itself. Second, while \cite{Cohn07} deals with continuous optimisation
problems, our optimisation problem is inherently of a discrete nature. It
could be strengthened into a continuous problem by letting $\frac{n_1
\cdots n_d}{2}$ particles move on the continuous torus $\RR^d/(n_1 \ZZ
\times \cdots \times n_d\ZZ)$. If the energy-minimising configurations for
that problem are (unique up to translation and) equal to our checkerboard
configurations, then that implies our results (provided that one uses
the $1$-norm, which is the continuous analogue of the Lee distance).
It would be interesting to know if the techniques from \cite{Cohn07}
can be used to prove such a stronger statement. Note that the pre-image
in $\RR^d$ of the particles' positions form a periodic set. The paper
\cite{Cohn07} does discuss optimisation problems for such sets, but in the
context where the lattice of periodicity is allowed to vary, as well. In
our case, the lattice is fixed to $n_1 \ZZ \times \cdots \times n_d\ZZ$.

Our minimisation problem is also reminiscent of certain instances of the
{\em quadratic assignment problem} that are studied in \cite{Burkard98}. In that
paper it is proved \jan{that the class of those instances is
NP-complete}. In view of this, minimality of checkerboard
configurations is an interesting result.  On the other hand, in
\cite{Burkard98} the corresponding {\em maximisation} problem is proved
to be solvable in polynomial time. In our setting, this would correspond
to taking for $u$ some {\em increasing} function of the distance, which
has an interpretation in terms of {\em attracting} particles rather than
repelling ones. This problem, in turn, resembles the problem of {\em how
to build an optimal city} \cite{Bender04}. Our attention,
however, is \jan{focused} entirely on the setting where $u$ is some {\em decreasing} function of the Lee distance.

\section{\jan{The one-dimensional case}} \label{sec:Circle}
The first special case \jan{that} we consider is the one-dimensional
toric grid. That is, we consider $G=(\ZZ/n)$, with $n$ a
positive integer, and define a graph on $G$ by connecting elements that differ by $\pm 1$. One can think of this as $n$ equally spaced points on a circle.
Further we define the function $u:G \times G \setminus \Delta \to \RR_+$ as $u(g,h):=f(\delta(g,h))$, where $\delta(g,h)$ denotes the Lee distance between $g$ and $h$ and $f:\RR_+ \to \RR_+$ is a decreasing strictly convex function. In~\cite{Gotz03} the continuous problem of letting particles move on a circle is considered and it is shown that total energy is minimised if and only if the particles are equally spaced. This validates the Checkerboard Conjecture for this special case.

The proof in~\cite{Gotz03} is based on a direct comparison
between the total energy in the configuration with equally
spaced particles (i.e.~the checkerboard configuration and
its translations) and the total energy in any other
configuration. To get a better feeling for the technique of
Section~\ref{sec:Framework} we \jan{now} assume $f(x)=x^{-1}$ and prove that fractional energy is uniquely minimised as well by the checkerboard vectors. That is, we will prove the following result.
\begin{thm} \label{thm:Circle}
Assume that $n$ is a multiple of $2$. Then the fractional energy $(x|Ax)$ over all $x \in V=\RR^G$ with
$(x|x)=(x|\one)=n/2$ is minimised by the characteristic vectors
$x_{S_\text{even}}$ and $x_{S_\text{odd}}$ of the checkerboard
configurations
\[ S_\text{even/odd}:=\{g \in G \mid g \text{
even/odd}\}, \]
and only by these. As a consequence, both the total energy and the maximal
energy among all subsets $S \subseteq G$ with $|S|=n/2$ are uniquely
minimised by $S_\text{even}$ and $S_{\text{odd}}$.
\end{thm}
By Proposition~\ref{prop:Eigenvalue} we know that set of optimal solutions of the fractional energy minimisation problem is determined by the smallest eigenvalue of the energy kernel $A$. As discussed in Section~\ref{sec:Framework} the eigenvalues are functions on the dual group $G^\vee=C_n$, with $C_n$ the group of complex $n$-th roots of unity. And the fractional energy is uniquely minimised by the checkerboard vectors if the eigenvalue has a unique minimum at $-1$, which lives in $G^\vee$ as $n$ is even. Thus the following proposition implies Theorem~\ref{thm:Circle}

\begin{prop}
The function $\lambda:G^\vee \to \RR$,
\[
\lambda(\chi) = \sum_{g \neq 0} \delta(g,0)^{-1} \chi^g
\]
has a unique minimum at $\chi=-1$.
\end{prop}
\begin{proof}
As $n$ is even we can write
\[
\lambda(\chi)=\sum_{k=1}^{n/2-1} \frac{1}{k}(\chi^{-k}+\chi^k) + \frac{2}{n}\chi^{n/2}.
\]
Now note that the $n$ elements of the group $C_n$ can be represented by $\e^{2\pi i j/n}$ for $j\in \mathcal{J} = \{0,\dots,n-1\}$. With a slight abuse of notation we can thus write
\[
\lambda(j)=\sum_{k=1}^{n/2-1} \frac{1}{k} 2 \cos (\frac{2\pi j k}{n}) + \frac{2}{n} (-1)^j.
\]
So, we need to prove that $\argmin\lambda(j)=\frac{n}{2}$.

For this first define $\kappa(x):[0,n)\rightarrow \mathbb{R}$ by
\[
\kappa(x)=\Re\Big[\sum_{k=1}^{n/2-1} \frac{1}{k} 2 \cos (\frac{2\pi x k}{n}) + \frac{2}{n} (-1)^x\Big].
\]
Now note that all eigenvalues are real numbers, so $\kappa(j)=\lambda(j)$ for all $j\in \mathcal{J}$. Hence, $\argmin \kappa(x)=\frac{n}{2}$ implies $\argmin\lambda(j)=\frac{n}{2}$.

To prove this let us look at the derivative of $\kappa(x)$. Differentiating $\kappa(x)$ we get
\[
\kappa '(x)=\frac{-4\pi}{n} \sum_{k=1}^{n/2-1} \sin (\frac{2\pi x k}{n}) - \frac{2\pi}{n}\sin (x\pi).
\]
To simplify this expression note that
\[
\sum_{k=1}^m \sin (kx) = \frac{\cos(\frac{x}{2})-\cos((m+\frac{1}{2})x)}{2\sin (\frac{x}{2})},
\]
as can be proved using the Prosthaphaeresis formulas.

We thus find
\begin{eqnarray*}
\kappa '(x)&=&\frac{-4\pi}{n} \frac{\cos(\frac{\pi x}{n})-\cos(\frac{(n-1)\pi x}{n})}{2\sin (\frac{\pi x}{n})}- \frac{2\pi}{n}\sin (\pi x)\\
&=&\frac{-4\pi}{n} \frac{\cos(\frac{\pi x}{n})-(\cos(\pi x)\cos (\frac{\pi x}{n})+\sin(\pi x)\sin (\frac{\pi x}{n}))}{2\sin (\frac{\pi x}{n})}- \frac{2\pi}{n}\sin (\pi x)\\
&=&\frac{2\pi}{n}(\cos(\pi x)-1)\frac{\cos (\frac{\pi x}{n})}{\sin (\frac{\pi x}{n})}.
\end{eqnarray*}
Hence, $\kappa '(x)\leq 0$ for $x \in (0,\frac{n}{2})$ and $\kappa '(x) \geq 0$ for $x \in (\frac{n}{2},n)$. Further, by l'H\^opital's rule, $\lim_{x\to 0} \kappa '(x)=0$. Also, $\kappa'(x) < 0$ for $x \in (\frac{n}{2}-1,\frac{n}{2})$ and $\kappa '(x) > 0$ for $x \in (\frac{n}{2},\frac{n}{2}+1)$. Thus $\kappa(x)$ has a unique minimum at $x^{*}=\frac{n}{2}$, which completes the proof.
\end{proof}

\section{\jan{The Hamming cube}} \label{sec:Hypercube}
In the previous section we considered the one-dimensional
toric grid with an arbitrary number of points. In this
section we \jan{consider} the $d$-dimensional toric grid, but
now we will assume that the grid has size two in all
dimensions. That is, we consider $G:=\prod_{i=1}^d
(\ZZ/2)=\{0,1\}^d$. Further we define a graph on $G$ by
connecting two elements if they differ by
\jan{$(0,\ldots,0,1,0,\ldots,0)$} for some position of the
$1$. One can think of this graph as the $d$-dimensional unit
hypercube; the corners of this hypercube are the points of
the toric grid. In fact, one should construct a $d$-torus
from the $d$-dimensional hypercube, but this only connects
corners that are connected already, as our grid has size two in every dimension.

Let $f:\jan{\ZZ_+} \to \RR_+$ be such that its $m$-th
forward difference has sign $(-1)^{m}$, i.e.\jan{,} we
assume that
\begin{equation}
\label{nforpos}
(-1)^m \Delta^m [f] (x)  >0
\end{equation}
for all $x\in\ZZ_+$. Here $\Delta$ denotes the forward difference,
$\Delta [f] (x) = f(x+1) -f(x)$.

We now define the function $u:G \times G \setminus \Delta \to \RR_+$ as
$u(g,h):=f(\delta(g,h))$, where $\delta(g,h)$ denotes the Lee distance between $g$ and $h$. We will
validate the Checkerboard Conjecture for this case. That is, we will prove the following theorem.
\begin{thm} \label{thm:Hypercube}
The fractional energy $(x|Ax)$ over all $x \in V=\RR^G$ with
$(x|x)=(x|\one)=|G|/2=2^{d-1}$ is minimised by the characteristic vectors
$x_{S_\text{even}}$ and $x_{S_\text{odd}}$ of the checkerboard
configurations
\[ S_\text{even/odd}:=\{(g_1,\ldots,g_d) \in G \mid \sum_i g_i \text{
even/odd}\}, \]
and only by these. As a consequence, both the total energy and the maximal
energy among all subsets $S \subseteq G$ with $|S|=2^{d-1}$ are uniquely
minimised by $S_\text{even}$ and $S_{\text{odd}}$.
\end{thm}
To prove Theorem~\ref{thm:Hypercube} we again capitalise on Proposition~\ref{prop:Eigenvalue} and determine the smallest eigenvalue of the energy kernel $A$. As the toric grid has size two in all dimensions we know by the discussion in Section~\ref{sec:Framework} that the eigenvalues are functions of the dual group $G^\vee=\{-1,1\}^d$. We further know that the fractional energy is minimised by the checkerboard vectors if the eigenvalue has a unique minimum at $(-1,\ldots,-1)\in G^\vee$. This we prove in the following proposition, which thus implies Theorem~\ref{thm:Hypercube}.
\begin{prop}
The function $\lambda : G^\vee \to \RR$,
\[ \lambda(\chi) = \sum_{g \neq 0} f(\delta(g,0)) \prod_{i=1}^d\chi_i^{g_i} \]
has a unique minimum at $\chi=(-1,\ldots,-1)$.
\end{prop}

\begin{proof}
\jan{
Because of symmetry it suffices to prove that if
$\chi,\tilde{\chi} \in \{-1,1\}^d$ are characters that differ only
in the first position, where $\chi_1=1,\tilde{\chi}_1=-1$, then
$\lambda(\chi)>\lambda(\tilde{\chi})$. We therefore compute
\begin{align*}
\lambda(\chi)-\lambda(\tilde{\chi}) &=
\sum_{g \neq 0} f(\delta(g,0)) (1-(-1)^{g_1}) \chi_2^{g_2} \cdots \chi_d^{g_d}\\
&=
2 \sum_{g=(g_2,\ldots,g_d)} f(\delta(g,0)+1) \chi_2^{g_2} \cdots
\chi_d^{g_d}.
\end{align*}
Writing $q$ for the number of ones among the
$\chi_2,\ldots,\chi_d$ and using
dummy $l$ for the number of indices $j$ with $\chi_j=g_j=1$
and dummy $k$ for the number of indices $j$ with
$\chi_j=-1$ and $g_j=1$ the above expression equals
\begin{align*}
&2 \sum_{l=0}^q
\binom{q}{l} \sum_{k=0}^{d-1-q} f(l+k+1) \binom{d-1-q}{k}
(-1)^k\\
&=2 \sum_{l=0}^q
\binom{q}{l} (-1)^{d-1-q} (\Delta^{d-1-q} f)(l+1)\\
&>0;
\end{align*}
where the last inequality follows from the conditions on $f$ and the
last equality follows from the well-known formula
\[
\Delta^m [f] (x) = \sum_{k=0}^m \binom{m}{k} (-1)^{m+k} f(x+k);
\]
see for instance \cite[page 203, equation 1a]{Rio79}. This
proves the proposition.}
\end{proof}

\leaveout{
Because of symmetry we have $\lambda(\hat\chi)=\lambda(\tilde\chi)$ whenever $\sum_{i=1}^d \hat\chi_i = \sum_{i=1}^d \tilde\chi_i$. We therefore consider the function $\kappa : \{0,\ldots,d\} \to \RR$ defined by
\[
\kappa(u) = \sum_{g \neq 0} f(\delta(g,0)) \prod_{i=1}^{d-u}(-1)^{g_i},
\]
so that $\kappa(u)=\lambda(\chi)$ whenever $\sum_{i=1}^d \chi_i = 2u-d$. The variable $u$ can be thought of as the number of ones in the vector $\chi$. In order to prove the proposition we thus have to prove that $\kappa(u)$ has a unique minimum at $u=0$.

Now first note that $\kappa(u)$ can be written as
\[
\kappa (u) = \sum_{a=0}^u\sum_{b=0}^{d-u} \binom{u}{a}\binom{d-u}{b}(-1)^b f(a+b) - f(0).
\]
Then,
\begin{eqnarray*}
\kappa (u+1) - \kappa(u) &=& \sum_{a=0}^{u+1}\sum_{b=0}^{d-u-1} \binom{u+1}{a}\binom{d-u-1}{b}(-1)^bf(a+b) - \sum_{a=0}^u\sum_{b=0}^{d-u} \binom{u}{a}\binom{d-u}{b}(-1)^bf(a+b) \\
&=& \sum_{c=0}^d f(c) \Big( \sum_{a=0}^{u+1} \binom{u+1}{a}\binom{d-u-1}{c-a} (-1)^{c-a} - \sum_{a=0}^u \binom{u}{a}\binom{d-u}{c-a}(-1)^{c-a}\Big) \\
&=& \sum_{c=0}^d f(c)(-1)^c \Big( \sum_{a=0}^{u+1} \binom{u+1}{a}\binom{d-u-1}{c-a} (-1)^{a} - \sum_{a=0}^u \binom{u}{a}\binom{d-u}{c-a}(-1)^{a}\Big).
\end{eqnarray*}
By Pascal's rule we know
\begin{eqnarray*}
\binom{u+1}{a} =\binom{u}{a} + \binom{u}{a-1}
\end{eqnarray*}
Further,
\begin{eqnarray*}
\sum_{a=0}^{u+1} \binom{u}{a-1}\binom{d-u-1}{c-a} (-1)^{a} &=& \sum_{a=0}^{u} \binom{u}{a}\binom{d-u-1}{c-a-1} (-1)^{a+1} \\
&=& \sum_{a=0}^{u} \binom{u}{a} \Big( \binom{d-u}{c-a} -\binom{d-u-1}{c-a} \Big) (-1)^{a+1}.
\end{eqnarray*}
Thus,
\begin{equation}
\label{kappadiff}
\kappa (u+1) - \kappa(u) = 2 \sum_{c=0}^d f(c)(-1)^c \sum_{a=0}^{u}\binom{u}{a}(-1)^{a}\Big( \binom{d-u-1}{c-a} - \binom{d-u}{c-a} \Big).
\end{equation}

We now need the following lemma.
\begin{lem}
\label{lem1}
$\sum_{i=0}^k \binom{k}{i} \binom{x-k}{y-i} (-1)^i = \sum_{i=0}^k \binom{k}{i} \binom{x-i}{y-i} (-2)^i$.
\end{lem}
\begin{proof}
Using~\cite[page 8, equation 5]{Rio79} we know,
\[
\binom{x-k}{y-i} = \sum_{j=0}^{k-i} (-1)^j \binom{k-i}{j} \binom{x-i-j}{y-i-j}.
\]
Thus,
\begin{eqnarray*}
\sum_{i=0}^k \binom{k}{i} \binom{x-k}{y-i} (-1)^i &=& \sum_{i=0}^k \binom{k}{i} (-1)^i  \sum_{j=0}^{k-i} (-1)^j \binom{k-i}{j} \binom{x-i-j}{y-i-j} \\
&=& \sum_{i=0}^k \binom{k}{i} \sum_{j=0}^{k-i} \binom{k-i}{j} \binom{x-i-j}{y-i-j} (-1)^{i+j} \\
&=& \sum_{j=0}^k \sum_{i=0}^{j} \binom{k}{i} \binom{k-i}{j-i} \binom{x-j}{y-j} (-1)^{j}.
\end{eqnarray*}
Hence, as
\[
\binom{k}{i}\binom{k-i}{j-i}=\binom{k}{j}\binom{j}{i}
\]
we get
\begin{eqnarray*}
\sum_{j=0}^k \sum_{i=0}^{j} \binom{k}{i} \binom{k-i}{j-i} \binom{x-j}{y-j} (-1)^{j} &=& \sum_{j=0}^k \binom{k}{j} \binom{x-j}{y-j} (-1)^{j} \sum_{i=0}^{j} \binom{j}{i} \\
&=&\sum_{j=0}^k \binom{k}{j} \binom{x-j}{y-j} (-1)^{j} 2^j.
\end{eqnarray*}
This completes the proof of the lemma.
\end{proof}

Applying Lemma~\ref{lem1} twice, equation~\eqref{kappadiff} simplifies to
\begin{eqnarray*}
\kappa (u+1) - \kappa(u) &=& 2 \sum_{c=0}^d f(c)(-1)^c \sum_{a=0}^{u}\binom{u}{a}(-2)^{a}\Big( \binom{d-a-1}{c-a} - \binom{d-a}{c-a} \Big) \\
&=& \sum_{c=0}^d f(c)(-1)^c \sum_{a=0}^{u}\binom{u}{a}(-2)^{a+1}\Big( \binom{d-a}{c-a} - \binom{d-a-1}{c-a} \Big).
\end{eqnarray*}

Thus, using Pascal's rule again,
\begin{eqnarray*}
\kappa (u+1) - \kappa(u) &=& \sum_{c=0}^d f(c)(-1)^c \sum_{a=0}^{u}\binom{u}{a}(-2)^{a+1} \binom{d-a-1}{c-a-1} \\
&=& \sum_{c=0}^d f(c+1) \sum_{a=0}^{u}\binom{u}{a} 2^{a+1} \binom{d-a-1}{c-a} (-1)^{c+a} \\
&=& \sum_{a=0}^{u} \sum_{c=0}^{d-a-1} \binom{u}{a} 2^{a+1} \binom{d-a-1}{c} (-1)^{c} f(a+c+1) \\
&=& \sum_{a=0}^{u} \binom{u}{a} 2^{a+1} \sum_{c=0}^{d-a-1}  \binom{d-a-1}{c} (-1)^{d-a-1-c}f(a+c+1) (-1)^{d-a-1}.
\end{eqnarray*}
Hence, by equation~\eqref{nfor}
\begin{eqnarray*}
\kappa(u+1) - \kappa(u) &=& \sum_{a=0}^{u} \binom{u}{a} 2^{a+1} \Delta^{d-a-1} [f] (i+1) (-1)^{d-a-1}.
\end{eqnarray*}
Thus, by equation~\eqref{nforpos}, $\kappa (u+1) - \kappa(u)>0$. So $\kappa(u)$ is strictly increasing and hence has a unique minimum at $u=0$. Thus, $\lambda(\chi)$ has a unique minimum at $\chi=(-1,\ldots,-1)$, which is what we had to prove.
\end{proof}
}

\section{Multiples of four} \label{sec:Four}

\jan{In this section we prove our strongest result about the Checkerboard
Conjecture, concerning $(n_1 \cdots n_d)/2$ particles on the Abelian
group $G=(\ZZ/n_1\ZZ) \times \cdots \times (\ZZ/n_d\ZZ)$ when each $n_i$
is either two or a multiple of four. Define a graph on $G$ by connecting
two elements if they differ by $\pm (0,\ldots,0,1,0,\ldots,0)$ for some
position of the $1$. The Lee distance is the shortest-path distance
between $g,h \in G$ in that graph, and denoted
by $\delta(g,h)$. The energy will be measured with a function $f:\RR_+
\to \RR_+$ that we assume to be (strictly) {\em completely monotonic},
i.e., $C^\infty$ and with $(-1)^k f^{(k)}(x)>0$ for all $x \in \RR_+$
and for all $k$.}

\begin{re}
Typical examples of completely monotonic $f$ are $f(x)=x^{-\alpha}$ for
$\alpha$ a positive real number. These functions appear prominently in the
paper \cite{Cohn07}, as well. Moreover, as mentioned in the introduction, these function generalise the inverse square law in physics, and describe the path loss for electromagnetic waves.
\end{re}

\begin{re}
\jan{Note that if $g:\RR_+ \to \RR_+$
is completely monotonic, then so is $-\Delta g$, because
\[ (-1)^k (-\Delta g)^{(k)} (x)=
(-1)^k(g^{(k)}(x)-g^{(k)}(x+1))=
(-1)^{k+1} \int_x^{x+1} g^{(k+1)}(t)\mathrm{d}t >0 \]
for all $x>0$. Iterating, we find that $(-\Delta)^k g$ is a
positive function, so that the restriction of $g$ to $\ZZ_+$
satisfies the condition on $f$. In particular,
Theorem~\ref{thm:Hypercube} implies the special case of the
upcoming
Theorem~\ref{thm:Fourfolds} where all $n_i$ are equal to
$2$.}
\end{re}

Define the function $u:G \times G \setminus \Delta \to
\RR_+$ as $u(g,h):=f(\delta(g,h))$. The maximal energy $\Emax$, the total
energy $\Etot$, the energy kernel $A$, and the fractional energy are
defined using this function $u$.

\begin{thm} \label{thm:Fourfolds}
Assume that each of $n_1,\ldots,n_d$ is either $2$ or a multiple of
$4$. Then the fractional energy $(x|Ax)$ over all $x \in V=\RR^G$ with
$(x|x)=(x|\one)=|G|/2$ is minimised by the characteristic vectors
$x_{S_\text{even}}$ and $x_{S_\text{odd}}$ of the checkerboard
configurations
\[ S_\text{even/odd}:=\{(j_1,\ldots,j_d) \mid \sum_i j_i \text{
even/odd}\}, \]
and only by these. As a consequence, both the total energy and the maximal
energy among all subsets $S \subseteq G$ with $|S|=|G|/2$ are uniquely
minimised by $S_\text{even}$ and $S_{\text{odd}}$.
\end{thm}

\jan{By Proposition~\ref{prop:Eigenvalue} it suffices to find
the smallest eigenvalue of the energy kernel $A$. Recall that the
eigenvalues are functions on the dual group $G^\vee$, which in
our present setting we may identify with $G^\vee=C_{n_1} \times \cdots
\times C_{n_d}$, where $C_a$ is the group of complex $a$-th roots of
unity. Since all $n_i$ are even, $G^\vee$ contains the real character
$-\one=(-1,\ldots,-1)$, with character space $V_{\{-\one\}}$ spanned by
the vector $z$ defined by $z(j_1,\ldots,j_d)=(-1)^{j_1+\ldots+j_d}$. The
checkerboard configurations have characteristic vectors $\frac{1}{2}\one
\pm \frac{1}{2}z$. Thus the following proposition implies
Theorem~\ref{thm:Fourfolds}.}

\begin{prop} \label{prop:Fourfolds}
The function $G^\vee \to \RR$ sending $\chi$ to
\[ \sum_{g \neq 0} f(\delta(g,0)) \chi(g) \]
has a unique minimum at $(-1,\ldots,-1)$.
\end{prop}

\begin{proof}
\jan{The completely monotonic function} $f$ can be written as
\[ f(x)=\int_{0}^\infty e^{-xt} \mathrm{d} \alpha(t), \]
where $\alpha$ is non-decreasing and the integral converges for $0 < x <
\infty$ (for this theorem of Bernstein see \cite[Chapter IV, Theorem
12b]{Widder41}). Thus, after interchanging integral and (finite) sum
the function that we seek to minimise \jan{is seen to equal}
\[ \chi \mapsto \int_{0}^\infty \left(\sum_{g \neq 0} e^{-\delta(g,0)t}
\chi(g) \right) \mathrm{d} \alpha(t). \]
For this function to have a unique minimum at $\chi=(-1,\ldots,-1)$
it suffices that for each fixed $t>0$ the function
\[ G^\vee \to \RR,\ \chi \mapsto \sum_{g \neq 0} e^{-\delta(g,0)t} \chi(g) \]
has a unique minimum at $\chi=(-1,\ldots,-1)$. At this point, there is
no harm in including $g=0$ in the sum, which adds a constant term $1$
independently of $\chi$. Writing $a:=e^t>1$, we want to minimise
\[ \sum_{g} a^{-\delta(g,0)} \chi(g) \]
over all characters $\chi$. Write $g=(g_1,\ldots,g_d)$; then \jan{the
Lee distance $\delta(g,0)$ equals $\delta(g_1,0)+\cdots+\delta(g_d,0)$, where we
use the notation $d$ also for the Lee distance in the individual
$\ZZ/n_i\ZZ$. Hence we have
\[ \sum_{g} a^{-\delta(g,0)} \chi(g)
= \sum_{(g_1,\ldots,g_d)} a^{-(\delta(g_1,0)+\ldots+\delta(g_d,0))}
\chi_1(g_1) \cdots \chi_d(g_d), \]
which factorises as the product of the factors
\[ \sum_{g_i \in \ZZ/n_i} a^{-\delta(g_i,0)} \chi_i(g_i) \]
for $i=1,\ldots,d$. Thus we are done if we can show
that each of these factors is positive for all $\chi_i$, and minimal
at $\chi_i=-1$.
}

What remains is an easy calculation. Simplify notation by fixing $i$
and writing $n:=n_i$. As a warm-up, consider the case where $n=2$.
Then the trivial character $1$ of $\ZZ/2\ZZ$ is mapped to $1+a^{-1}$,
while the character $-1$ of $\ZZ/2\ZZ$ is mapped to $1-a^{-1}$. The
latter value is smaller than the former, as desired. Next assume that $n$
is divisible by $4$, and consider the function $\lambda: C_n \to \RR$
sending an $n$-th root of unity $\zeta$ to
\begin{align*}
&\sum_{g \in \ZZ/n\ZZ} a^{-\delta(g,0)} \zeta^g \\
&=1 + \sum_{g=1}^{n/2-1} a^{-g}(\zeta^g + \zeta^{-g}) + a^{-n/2} \zeta^{n/2}\\
&=-1 + \sum_{g=0}^{n/2-1} a^{-g}(\zeta^g + \zeta^{-g}) + a^{-n/2} \zeta^{n/2}\\
&=\frac{1-a^{-n/2}\zeta^{n/2}}{1-a^{-1}\zeta}
+
\frac{1-a^{-n/2}\zeta^{-n/2}}{1-a^{-1}\zeta^{-1}} -(1-a^{-n/2}\zeta^{n/2}).\\
\end{align*}
Now $\zeta^{n/2}$ equals $\pm 1$. If it equals $1$, then this expression
reduces to
\begin{align*}
&(1-a^{-n/2})\left(\frac{1}{1-a^{-1}\zeta}+\frac{1}{1-a^{-1}\zeta^{-1}}-1\right)
&=(1-a^{-n/2})\left(\frac{1-a^{-2}}{||1-a^{-1}\zeta||^2}\right).
\end{align*}
Since $a>1$, this is a positive number for any $\zeta$, and its smallest
value is attained for the $n/2$-th root of unity $\zeta_0$ furthest away
from $1$ in the complex plane, which equals $-1$ since $n$ is a multiple
of $4$. Next, if $\zeta^{n/2}$ equals $-1$, then the expression above
reduces to
\begin{align*}
&(1+a^{-n/2})\left(\frac{1}{1-a^{-1}\zeta}+\frac{1}{1-a^{-1}\zeta^{-1}}-1\right)
&=(1+a^{-n/2})\left(\frac{1-a^{-2}}{||1-a^{-1}\zeta||^2}\right),
\end{align*}
which is again positive for all $\zeta$ with $\zeta^{n/2}=-1$ and minimal
for the $\zeta_1$ furthest away from $1$ in the complex plane (which
equals $-e^{\pm 2 \pi i/n}$).

It remains to compare $\lambda(\zeta_0)$ and $\lambda(\zeta_1)$. For
this, note that the factor $(1-a^{-n/2})$ in the expression for
$\lambda(\zeta_0)$ is smaller than the factor $(1+a^{-n/2})$ in
the expression for $\lambda(\zeta_1)$, while moreover $\zeta_0$ is
further away from $1$ than $\zeta_1$. Hence the minimum is attained
at $\zeta=\zeta_0=-1$, as desired. This concludes the proof of the
proposition, and hence the proof of Theorem \ref{thm:Fourfolds}.
\end{proof}

\begin{re}
The problem in extending this argument to the case of general even $n_i$
is that in the proof of the proposition, when $n/2$ is odd and at least
$3$, $\lambda(\zeta_1)=\lambda(-1)$ can very well be {\em larger} than
$\lambda(\zeta_0)$. In other words,
\jan{Proposition~\ref{prop:Fourfolds}}
does not extend to that situation for general completely monotonic
functions $f$.  Extensive numerical evidence with $d=2$ and $n_1,n_2$
up to $100$ suggests that the proposition {\em does} extend to the
case where $f(x)=1/x$ (and probably even for $f(x)=x^{-\alpha}$ with
$\alpha>0$), in which case also
\jan{Theorem~\ref{thm:Fourfolds}} extends.

Of course, \jan{the Checkerboard Conjecture may still very well hold}
for general completely monotonic $f$ and general even $n_i$, even when
Proposition~\ref{prop:Fourfolds} does not. To prove this, one would
like to have a combinatorial/algebraic argument why the eigenspaces
corresponding to characters $\chi$ having some entries equal to the
relevant $-\exp(2 \pi i/n)$ with $n/2$ odd are no good for finding
$0/1$-vectors. At this moment we have no such argument.  \end{re}

\begin{figure}
\includegraphics[scale=.5]{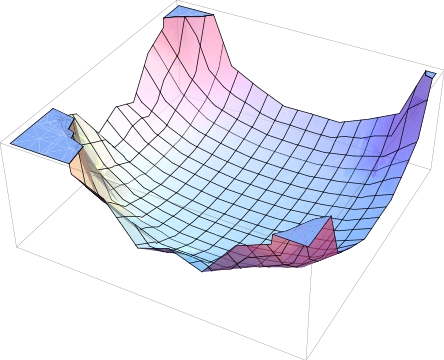}
\caption{Eigenvalues of the energy kernel (interpolated to make an
easily interpretable surface; and truncated near the corners).}
\label{fig:n10}
\end{figure}

\begin{re} \label{re:DFT}
Note that the table of eigenvalues of $A$ contains $|G|$ entries, and
evaluating each individually involves a summation over $|G|$, so that
a straightforward implementation would require ${\mathcal O}(|G|^2)$
(floating-point) operations. However, the table of eigenvalues is nothing
but the {\em Fourier transform} of the table of values $f(\delta(g,0))$,
so that algorithms for ($d$-dimensional) {\em Fast Fourier Transform}
may be used to compute the table of eigenvalues more efficiently.
We have used {\tt Mathematica}'s built-in function {\tt Fourier} in
our numerical experiments. For instance, Figure \ref{fig:n10} shows
the eigenvalues of the energy kernel for $d=2$ and $n_1=n_2=10$ and
$f(x)=x^{-0.3}$ as a function of $(j,k)$, standing for the character
$(e^{\frac{2\pi i j}{10}},e^{\frac{2\pi i k}{10}})$. The ``rough edges''
are related to the even/odd effect in the proof above, but they smooth
out in the middle, so that the minimum is attained at $j=k=5$, which
corresponds to the character $(-1,-1)$.
\end{re}

\section{Open problems} \label{sec:Open}
In this paper we have proved the Checkerboard Conjecture for
toric grids and energy functions $f$ for which the unique minimal
eigenvalue of the energy kernel $A$ is uniquely attained at the character
$(-1,\ldots,-1)$ in the dual group. Indeed, the minimal eigenvalue of $A$
{\em always} governs the solution to fractional energy minimisation
(Proposition~\ref{prop:Eigenvalue}), and if it is uniquely attained at
$(-1,\ldots,-1)$, then the minimisers of fractional energy minimisation
are $0/1$-vectors and hence minimisers of the original
problem. This motivates the following open problems.

\subsection{Generalising Theorem \ref{thm:Fourfolds}}
One would like to extend the results of Section~\ref{sec:Four}
to the case where some of the $n_i$ are larger than $2$ but equal to
$2$ modulo $4$. Computer experiments with two-dimensional grids up to
$100 \times 100$ suggest that for $f(x)=1/x$ (and perhaps for general
inverse power laws) the minimal eigenvalue of the energy kernel is still
uniquely attained at $(-1,\ldots,-1)$. However, a proof of this fact
would require more sophisticated analytical methods than the factoring
into one-dimensional instances that we used there.  Alternatively, one
might try and find a combinatorial argument why the minimisers of total
energy minimisation are necessarily perpendicular to the eigenspaces
of $A$ corresponding to the ``undesirable'' characters, thus sharpening
the relaxation.

\subsection{Euclidean or other metrics}
The question remains for which metrics other than the Lee distance the
Checkerboard Conjecture will hold.  For the Euclidean metric the analogue
of Proposition~\ref{prop:Fourfolds} does not hold for general completely
monotonic functions $f$. The first steps of its proof go through when one
replaces the Lee distance with the square of the Euclidean distance. This
holds in particular for the all-important factorisation step.
One is then lead to minimise
\[ \sum_{g \in \ZZ/n} a^{-\delta(g,0)^2} \zeta^g \]
over the $n$-th roots $\zeta$ of unity, where $\delta$ stands
for the Euclidean metric on the one-dimensional toric grid, and
where $a$ is a real number greater than $1$. However, for $a$
close to $1$ the minimum is {\em not} in general attained at $-1$;
see Figure~\ref{fig:n8euclid}. On the other hand, extensive numerical
evidence for $2$-dimensional grids up to $100 \times 100$ suggests that
the analogue of Proposition~\ref{prop:Fourfolds} does hold when the energy
is given by an inverse power law of the Euclidean distance (and if these
computations were done in suitable interval arithmetic, they would {\em
prove} the Checkerboard Conjecture for those grids). Perhaps one can
prove such a statement by bounding the contribution, in the Bernstein
expansion of these inverse power laws, of the bit where $a$ is small.

\begin{figure}
\includegraphics[width=.5\textwidth]{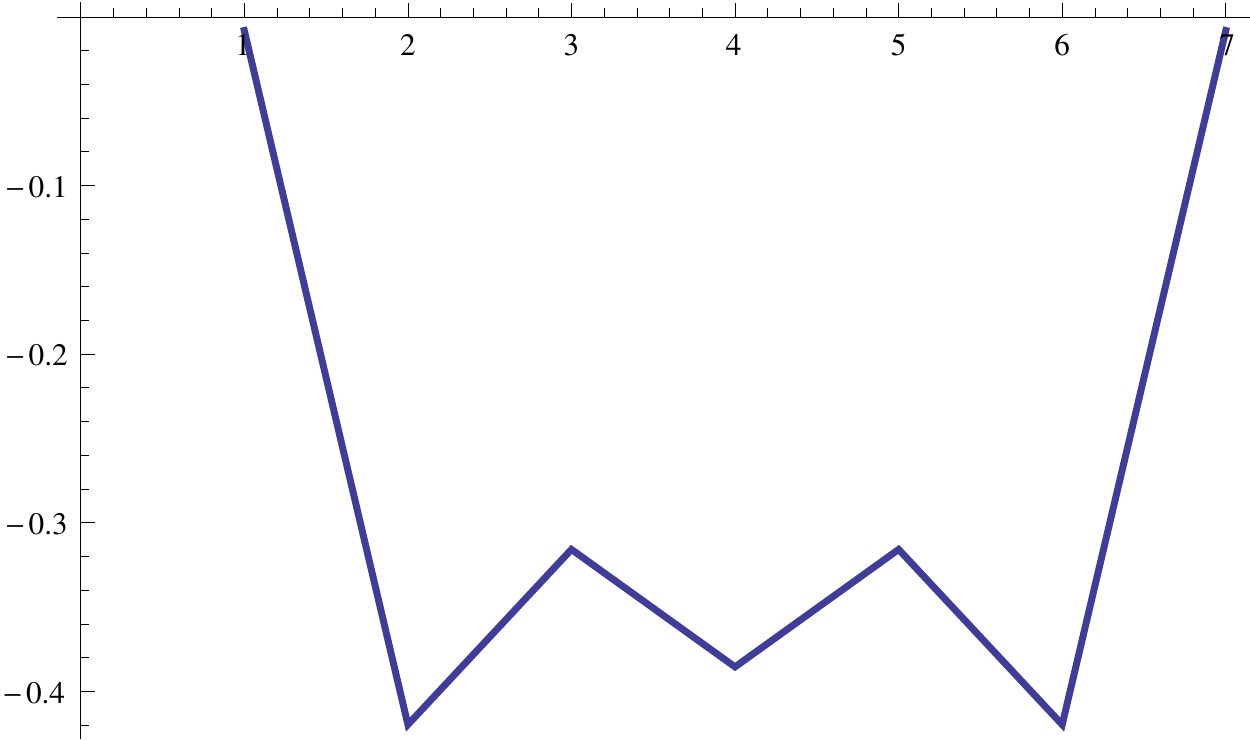}
\caption{Values (vertical axis) of $\sum_{g \in \ZZ/8}
(1.05)^{-\delta(g,0)^2} \zeta^g$ for $\zeta=\exp(\frac{2 k \pi
i}{8}),\ k=1,\ldots,7$ (horizontal axis); $k=0$ is excluded
to make it clear where the minimum is attained: at $k=2,6$
rather than $k=4$.}
\label{fig:n8euclid}
\end{figure}

The Checkerboard Conjecture is not true for {\em every} metric however. For
example, for the Chebyshev (or $L_\infty$) metric one can compute that
the maximal energy of the patterns in Figures~\ref{square2ndbest}
and~\ref{square3rdbest} is lower than the maximal energy of the
checkerboard pattern in a $6\times6$-grid. It is unknown what properties
a metric should have in order for the Checkerboard Conjecture to be true.

\subsection{Nearly optimal configurations}
It is also interesting to consider {\em nearly optimal} configurations
of $n/2$ particles in toric grids. When considering {\em total energy}
$\Etot$, these configurations are probably given by small perturbations
of checkerboard patterns.  However, nearly optimal configurations
for {\em maximal energy} $\Emax$ seem to exhibit beautiful, regular
patterns. For example, for a $6\times6$-grid we have computed the
maximal energy of all possible $\binom{36}{18}$ configurations for
the function $f(x)=\frac{1}{x}$ and for both the Lee metric and
the Euclidean metric.  This computation shows that the Checkerboard
conjecture holds (a fact that one can prove more efficiently using the
discrete Fourier transform as in Remark~\ref{re:DFT}), but also that
the second and third best configurations agree for the Lee and Euclidean
metrics (see Figures~\ref{square2ndbest} and~\ref{square3rdbest}). The
fourth best configurations are different for the two metrics though,
see Figures~\ref{square4thbestLee} and~\ref{square4thbestEuclid}. It
is a widely open question how to prove that certain configurations are
second or third best.

\begin{figure}
\hspace{0.2in}
\subfigure[Second best for Lee and Euclidean metric.]{\includegraphics[width=1.5in,keepaspectratio]{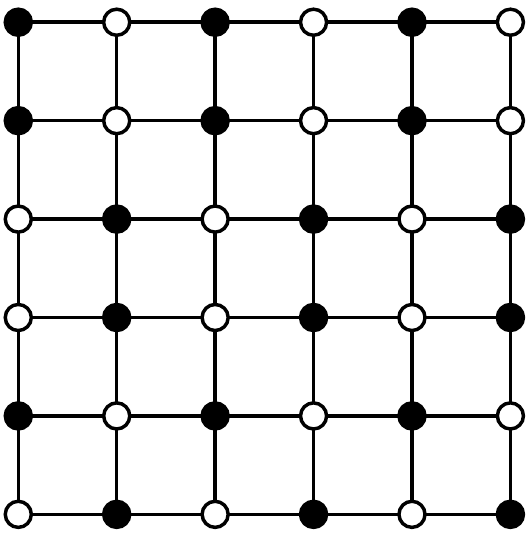}
  \label{square2ndbest}}
  \hspace{0.8in}
\subfigure[Third best for Lee and Euclidean metric.]{
  \includegraphics[width=1.5in,keepaspectratio]{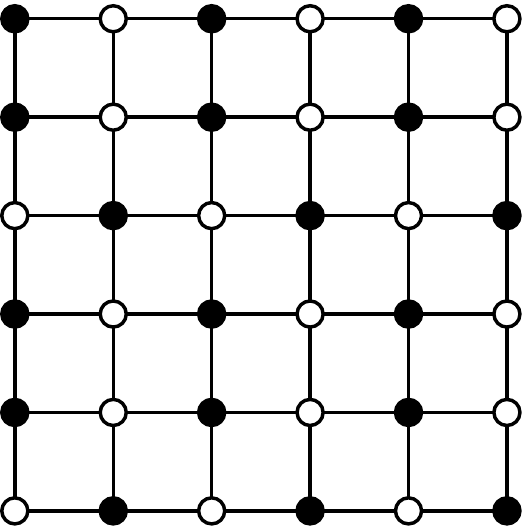}
  \label{square3rdbest}}
  \\
  \hspace{0.2in}
\subfigure[Fourth best for Lee metric.]{\includegraphics[width=1.5in,keepaspectratio]{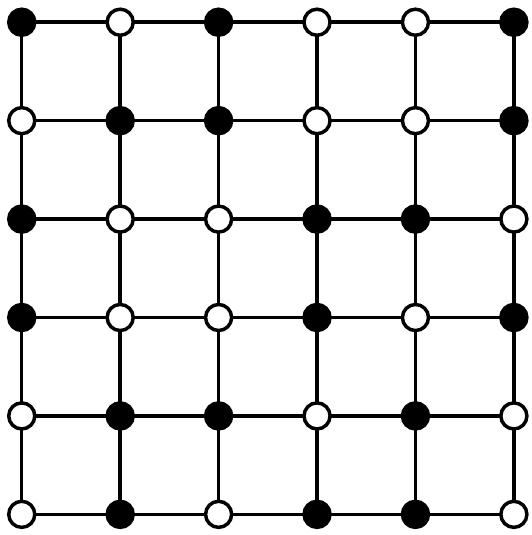}\label{square4thbestLee}}
  \hspace{0.8in}
\subfigure[Fourth best for Euclidean metric.]{
  \includegraphics[width=1.5in,keepaspectratio]{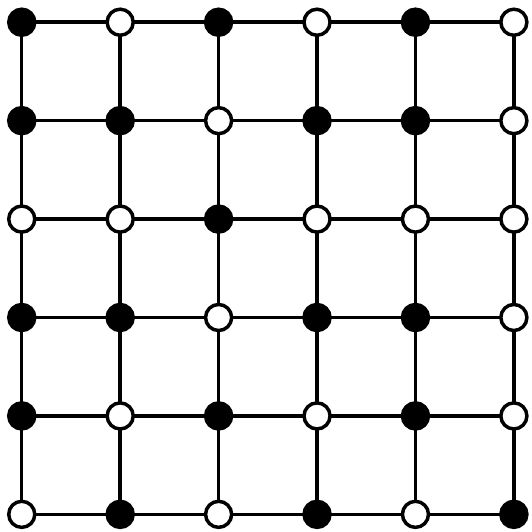}
  \label{square4thbestEuclid}}
\caption{The black nodes denote the nearly optimal
configurations for $\Emax$ in a $6 \times 6$-grid for $f(x)=\frac{1}{x}$.}
\end{figure}


\end{document}